\documentclass[preprint]{elsarticle}

\usepackage{amsmath,amssymb,amsthm}
\usepackage{graphicx,enumerate}

\newtheorem{thm}{Theorem}
\newtheorem{lem}{Lemma}
\newtheorem*{mrthm}{Theorem (Morales-Ramis)}
\newtheorem*{kthm}{Theorem (Kovacic)}

\newcommand{\bmt}[1]{\mbox{\boldmath $#1$}}

\begin{document}

\title{Non-integrability of geodesic flow on certain algebraic surfaces}

\author{T. J. Waters}
\ead{thomas.waters@port.ac.uk}

\address{Department of Mathematics, University of Portsmouth,
Portsmouth PO13HF, United Kingdom}

\begin{abstract}
This paper addresses an open problem recently posed by V. Kozlov:
a rigorous proof of the non-integrability of the geodesic flow on
the cubic surface $x y z=1$. We prove this is the case using the
Morales-Ramis theorem and Kovacic algorithm. We also consider some
consequences and extensions of this result.
\end{abstract}

\maketitle

\section{Introduction}

In two recent papers \cite{kozlov1} and \cite{kozlov2} Kozlov
posed the following open problem: to rigorously prove the
non-integrability (in the sense of Louiville) of the geodesic flow
on the surface $x y z=1$. In what follows we will exploit the
Hamiltonian nature of the geodesic equations by examining the
variational equations about a planar geodesic. The crucial theorem
we shall make use of is due to Morales-Ruiz and Ramis, which we
quote from \cite{morruiz}:

\begin{mrthm} For a $2n$ dimensional Hamiltonian system assume there are $n$ first integrals which are meromorphic,
in involution and independent in the neighbourhood of some
non-constant solution. Then the identity component of the
 differential Galois group of the normal variational equation (NVE) is an abelian subgroup of the symplectic group.
\end{mrthm}

On two-dimensional manifolds the normal variational equation
(which we shall derive in the next section) is a second order
linear ordinary differential equation with meromorphic
coefficients. In particular, we will see that the NVE's of
interest are Fuchsian. To show that the geodesic flow is not
meromorphic-integrable it suffices to show that the NVE is not
solvable in the sense of differential Galois theory: the identity
component of the differential Galois group of the NVE is not
abelian. This means we cannot ``build'' the solutions from the
field of meromorphic functions by adjoining integrals,
exponentiation of integrals, or algebraic functions of elements of
the field of meromorphic functions. To test this we make use of
the Kovacic algorithm.

Before we state the Kovacic algorithm we note that this algorithm
is very robust and can treat any second order linear ODE with
rational coefficients, however in the present work we will only
need a very limited portion of the algorithm. Thus to save space
we will present a much abbreviated version and refer the reader to
the original article of Kovacic \cite{kovacicmain} and the reduced
version appropriate for Fuchsian ODE's in Churchill and Rod
\cite{church}, whose notation we will follow most closely below.

Consider a linear ode of the following form \begin{align}
\frac{d^2\xi}{dy^2}=\xi''=r(y)\xi, \quad r(y)\in\mathbb{C}(y).
\label{maineq}
\end{align} If the equation is Fuchsian, that is it admits only
regular singular points, then we can decompose $r(y)$ as
\begin{align*} r(y)=\sum_{j=1}^k\frac{\beta_j}{(y-a_j)^2}+\sum_{j=1}^k
\frac{\delta_j}{y-a_j}, \end{align*} where $k$ is the number of
finite regular singular points at locations $y=a_j$. When $\sum
\delta_j=0$ then $y=\infty$ is also a regular singular point, with
$\beta_{\infty}=\sum(\beta_j+\delta_j a_j)$. The indicial
exponents are \begin{align*}
\tau_j^\pm=\tfrac{1}{2}\left(1\pm\sqrt{1+4\beta_j}\right),\quad
\tau_\infty^\pm=\tfrac{1}{2}\left(1\pm\sqrt{1+4\beta_\infty}\right)
\end{align*} at $y=a_j$ and $y=\infty$ respectively.

Kovacic proved in \cite{kovacicmain} that there are only 4
possible cases for the differential Galois group of
\eqref{maineq}. We will see in Section 3 that we can rule out two
of these cases immediately (cases $I$ and $III$), and as such we
will present only necessary conditions for these cases.

\begin{kthm}
Let $\mathcal{G}$ be the differential Galois group associated with
\eqref{maineq}, and note $\mathcal{G}\subset SL(2,\mathbb{C})$.
Then only one of four cases can hold:

\begin{enumerate}[(I)]
\item $\mathcal{G}$ is triangulisable (or reducible). A necessary condition for this case to hold is that, defining the
`modified exponents' $\alpha^\pm$ as \begin{align*}
&\alpha_j^\pm=\tau_j^\pm\ \textrm{if}\ \beta_j\neq 0;\quad
\alpha_j^\pm=1\ \textrm{if}\ \beta_j=0\ \textrm{and}\ \delta_j
\neq 0;\quad \alpha_j^\pm=0\ \textrm{if}\ \beta_j=\delta_j=0, \\
&\alpha_\infty^\pm=\tau_\infty^\pm\ \textrm{if}\ \beta_\infty\neq
0; \quad \alpha_\infty^+=1,\alpha_\infty^-=0\ \textrm{if}\
\beta_\infty=0,
\end{align*} there is some combination \begin{align*}
d=\alpha_\infty^\pm-\sum_{j=1}^k\alpha_j^\pm \, \in
\mathbb{N}_0=0,1,2,3,\ldots
\end{align*}
\item $\mathcal{G}$ is conjugate to a subgroup of the `DP' group, in the terminology of Churchill and Rod. A
necessary and sufficient condition for this case to hold is that,
defining the following sets \begin{align}
&E_j=\{2+e\sqrt{1+4\beta_j},e=0,\pm 2\}\cap\mathbb{Z}\
\textrm{if}\ \beta_j\neq 0; \nonumber \\
&E_j=\{4\}\ \textrm{if}\ \beta_j=0,\delta_j\neq 0;\quad E_j=\{0\}\
\textrm{if}\ \beta_j=\delta_j=0; \nonumber \\
&E_\infty=\{2+e\sqrt{1+4\beta_\infty},e=0,\pm 2\}\cap\mathbb{Z}\
\textrm{if}\ \beta_\infty\neq 0;  \nonumber \\
&E_\infty=\{0,2,4\}\ \textrm{if}\ \beta_\infty=0, \label{esets}
\end{align}  there is some combination of $e_j\in E_j$ and
$e_\infty\in E_\infty$, not all even integers, so that
\begin{align*} d=\frac{1}{2}\Big(e_\infty-\sum_{j=1}^k e_j
\Big)\,\in\mathbb{N}_0 \end{align*} and there exists a monic
polynomial $P(y)$ of degree $d$ which solves the following ODE:
\begin{align} \hspace{-1.1cm} P'''+3\theta
P''+(3\theta^2+3\theta'-4r)P'+(\theta''+3\theta\theta'+\theta^3-4r\theta-2r')P=0,\quad
\theta=\frac{1}{2}\sum_{j=1}^k \frac{e_j}{y-a_j}. \label{plong}
\end{align}
\item $\mathcal{G}$ is finite. A necessary condition for this case to hold is that all indicial exponents
$\tau_j^{\pm}$ and $\tau_\infty^{\pm}$ are rational.
\item $\mathcal{G}=SL(2,\mathbb{C})$, whose identity component is not abelian and therefore \eqref{maineq} is
not solvable.
\end{enumerate}
\end{kthm}

In the next section we will derive the NVE about a planar geodesic
on the Monge patch with a plane of symmetry; in Section 3 we will
prove, using the theorems presented in this Section, that the NVE
is not solvable and therefore the geodesic flow is not Louiville
integrable on the surface $xyz=1$. In Section 4 we will consider
some extensions and consequences of this result, and in Section 5
we finish with some conclusions.

We note that the approach followed in this paper has been used to
prove the non-integrability of a number of problems in mechanics
and celestial mechanics (see, for example, \cite{mrthm1},
\cite{mrthm2}, \cite{bouch}, \cite{tsy}, \cite{humanez},
\cite{simon}, \cite{sam}, \cite{bardin}, to name but a few), but
with the exception of another work by the author \cite{TW} this
approach is novel in examining geodesic flow.


\section{Derivation of the NVE}

The key feature of the surface $xyz=c$, where w.l.o.g. we can set
$c=1$, which facilitates this analysis is that by a simple
rearrangement and rotation of $\pi/4$ about the $z$-axis we can
write $z=1/(x^2-y^2)$, or more generally \begin{align}
z=f(x,y),\quad f_{,x}(0,y)=0.\label{monge} \end{align} This means
the surface is a Monge patch (or graph) with a plane of symmetry
(or invariant plane), the $y$-$z$ plane. The surface is actually
made up of 4 identical components, and to demonstrate the
non-integrability of the geodesics of the surface we need only
demonstrate this property on one component. We will restrict our
attention to the quadrant
\begin{align*} \{x,y,z\in\mathbb{R}^3:|x|>|y|,x>0\} \end{align*} which
immediately rules out any possible divergences in $f$.

To keep the approach of this section general and to facilitate the
analysis of Section 4 we will derive the NVE on the Monge patch
with a plane of symmetry as in \eqref{monge}.

\begin{lem}
The normal variational equation about the planar geodesic on the
Monge patch \eqref{monge} is \begin{align*}
\ddot{\xi}+\mathcal{K}|_0\xi=0,
\end{align*} where $\mathcal{K}|_0$ is the Gauss curvature evaluated along
the planar geodesic.
\end{lem}

\begin{proof}
Using the standard parameterisation $(x,y,f(x,y))$ and resulting
line element we may calculate the Christoffel symbols and geodesic
equations: \begin{align*}
\ddot{x}^a+\sum_{b,c}\Gamma^a_{bc}\,\dot{x}^b\dot{x}^c=0, \quad
x^a=(x,y)
\end{align*} where a dot denotes differentiation w.r.t.\ arc-length $s$.
The $x=0$ plane is invariant since $\Gamma^x_{yy}=0$ when
$f_{,x}(0,y)=0$, and thus there is a planar geodesic
$(0,\tilde{y}(s),f(0,\tilde{y}(s)))$ where $\tilde{y}$ solves
$(1+f_{,y}(0,\tilde{y})^{2})\dot{\tilde{y}}^2=1$. Linearizing the
geodesic equations about this planar geodesic the normal
variational equation will simply be the variation in the
$x$-direction, namely
\begin{align*} \ddot{\xi}+\Big(\Gamma^x_{yy,x}\dot{\tilde{y}}^2\Big)\Big|_0 \xi=
\ddot{\xi}+\left(\frac{f_{,xx}
f_{,yy}}{(1+f_{,y}^2)^2}\right)\Big|_0
\xi=\ddot{\xi}+\mathcal{K}|_0\xi=0.
\end{align*}
\end{proof}

The problem with this equation is that the coefficient is a
function of $\tilde{y}(s)$, which is defined implicitly as a
solution of $(1+f_{,y}^{\ 2})\dot{\tilde{y}}^2=1$. Clearly
$\tilde{y}$ will also parameterise the planar geodesic and thus we
make the change of independent variable to $\tilde{y}$,
calculating derivatives such as (dropping the tildes)
\begin{align*}
\frac{d}{ds}=\dot{y}\frac{d}{dy}=\frac{1}{\sqrt{1+f_{,y}^{\
2}}}\frac{d}{dy}
\end{align*} and so on, to arrive at the NVE (dropping the 0 subscript) \begin{align}
\xi''-\left(\frac{f_{,y}f_{,yy}}{1+f_{,y}^{\
2}}\right)\xi'+\left(\frac{f_{,yy}f_{,xx}}{1+f_{,y}^{\
2}}\right)\xi=0. \label{nve}
\end{align}

If, for a given surface $z=f(x,y)$ with $f_{,x}(0,y)=0$, this NVE
is not solvable in the sense of differential Galois theory as
described in Section 1, then the geodesic flow on that surface is
not integrable. This is precisely what we will show in the next
Section for the surface $xyz=1$. Before we do however, we can make
a comment about \eqref{nve}:

Notice that the equation is of the form
$\xi''-f_{,y}Q(y)\xi'+f_{,xx}Q(y)\xi=0$. If $f_{,xx}/f_{,y}=1/y$,
then \eqref{nve} would have the simple solution $\xi_1=c_1 y$ from
which we could construct a second solution via integrals and
exponentiation of integrals of meromorphic functions. This leads
us to consider solutions of the PDE
\begin{align} y f_{,xx}-f_{,y}=0 \label{pde}
\end{align} (where we evaluate the derivatives of $f$ at
$x=0$) as candidates for surfaces with integrable geodesic flow.
Examples include well-known integrable surfaces such as
$f(x,y)=f(x^2+y^2)$ and more interesting surfaces such as
$f=\cos(\omega x) e^{-\omega^2 y^2/2}$. But we should not divert
too much attention to \eqref{pde}: a surface which solves this
equation need not have integrable geodesic flow, it merely passes
this integrability test.


\section{Non-integrability of the surface $xyz=1$}

In the case of $z=f(x,y)=(x^2-y^2)^{-1}$, the NVE takes the form
\begin{align} \xi''-\frac{18(2+3y^6)}{y^2(y^6+4)^2}\xi=0,
\label{nvexyz}
\end{align}where we have removed the $\xi'$ term from \eqref{nve} via
the standard transformation \cite{kovacicmain}, and we extend the
independent variable to the complex domain. There are 8 regular
singular points, $a_j=\{0,\rho_1,\ldots,\rho_6,\}$ and $\infty$
where $\rho_i$ denotes the 6 roots of $y^6+4=0$ symmetrically
distributed about the circle of radius $^6\sqrt{4}$ centred on the
origin. We find the $\beta$ coefficients are \begin{align*}
\beta_j=\Big\{-\frac{9}{4},\frac{5}{16},\ldots,\frac{5}{16}
\Big\},\quad\beta_\infty=0,
\end{align*} and only $\delta_\infty=0$. We can see immediately
that $\tau_0^\pm=\tfrac{1}{2}(1\pm i\sqrt{8})$, and therefore case
$III$ of the Kovacic algorithm can be ruled out (the finite case).
What's more, none of the other $\tau_j^\pm$ are complex so case
$I$ of the algorithm can also be ruled out (the triangulisable
case).

Case $II$ is more problematic. The sets described in \eqref{esets}
are \begin{align*} E_0=\{2\},\quad E_{1\ldots 6}=\{2,5,-1\},\quad
E_\infty=\{0,2,4\}.
\end{align*} There are 21 combination of the elements of these sets leading to each of $d=0$ and
$d=1$, and 1 leading to each of $d=2,3,4$ (for example,
$d=\tfrac{1}{2}\big(4-(2-1-1-1-1-1-1)\big)=4$). For each of these
combinations we must attempt to construct a monic polynomial of
order $d$ that satisfies \eqref{plong}. This can be done using a
computer algebra system such as Mathematica; the calculations are
tedious rather than difficult. By checking each combination we can
see that there is no polynomial $P$ satisfying \eqref{plong}. We
can now state the main result of this paper:

\begin{thm} The geodesic flow on the surface $x y z=1$ is not
integrable in the sense of Louiville with meromorphic first
integrals.
\end{thm}

\begin{proof}
The differential Galois group of the normal variational equation
\eqref{nvexyz} does not fall into case $I$, $II$ or $III$ of
Kovacic's algorithm, as we have shown above. Therefore we must
have $\mathcal{G}=SL(2,\mathbb{C})$, the identity component of
which (also $SL(2,\mathbb{C})$) is not abelian. By the
Morales-Ramis theorem of Section 1 this means the geodesic
equations are not Louiville integrable with meromorphic first
integrals.
\end{proof}


\section{Extensions and limitations}

It seems natural to ask can we use the same techniques to examine
other surfaces similarly defined. We will consider two
generalisations, $x^ny^nz^n=1$ and $x^ny^nz=1$.

\subsection{Surfaces of the form $x^ny^nz^n=1$}

While it might seem ``obvious'' that $xyz=1$ and $(xyz)^n=1$ are
``the same'', care needs to be taken. If $n$ is an even integer
then the surface will have twice as many components as when $n$ is
odd; for example the point $(1,1,-1)$ is on $x^2y^2z^2=1$ but not
on $xyz=1$. To show they are isometric would require the
calculation of the first fundamental form, which is not
well-defined for algebraic surfaces, i.e.\ surfaces defined
implicitly by $F(x,y,z)=c$. We can calculate the Gauss curvature
using the following expression \cite{spivak} (here $\nabla F$ and
$H(F)$ are the gradient and Hessian of $F$ respectively, and the
norms are w.r.t.\ the ambient Euclidean space)
\begin{align*} K=-\frac{\left|\begin{array}
  {cc} H(F) & \nabla F \\ \nabla F^T & 0
\end{array} \right|}{|\nabla F|^4}
\end{align*} which we find to be independent of the value of $n$,
but having the same Gauss curvature at identified points is a
necessary but not sufficient condition for isometry. Instead, we
can generate the geodesic equations themselves on the algebraic
surfaces in question, and prove the following theorem.

\begin{thm}
  The geodesic flow on the algebraic surface $x^n y^n z^n=1$ with
  $n\in\mathbb{R}$ is not Louiville integrable with meromorphic
  first integrals.
\end{thm}

\begin{proof}
  The geodesic equations on the algebraic surface $F(\bmt{r})=c$
  where $\bmt{r}=(x,y,z)$ are given by \cite{kozlov1} \begin{align*}
    \ddot{\bmt{r}}=\lambda \nabla F,\quad \lambda=-\frac{\big(H(F)\dot{\bmt{r}}\big).\dot{\bmt{r}}}{|\nabla
    F|^2}.
  \end{align*} Taking $F=x^ny^nz^n$ we find the geodesic equations
  are independent of $n$, i.e.\ the geodesic equations are the
  same for all values of $n$ (we need to make use of the fact that $\dot{F}=\nabla F.\dot{\bmt{r}}=0$). We have shown the geodesic equations
  are not Louiville integrable when $n=1$ in the previous Section,
  and therefore they are not integrable for $n\in\mathbb{R}$.
\end{proof}

\subsection{Surfaces of the form $x^ny^nz=1$}

It might be hoped that we could generalize the surface considered
in Section 3 to Monge patches of the form \begin{align}
z=\frac{1}{(x^2-y^2)^n}, \quad n\in\mathbb{N}.\label{gensur}
\end{align} Unfortunately the techniques described in this paper
do not allow for a uniform treatment, for the following reason.

Using the methods of Section 2 the NVE of the planar geodesic of
\eqref{gensur} is \begin{align*}
\xi''+\frac{2n^2(2n+1)(4n^3-10n^2-y^{4n+2}(4n+5))}{y^2(4n^2+y^{4n+2})^2}\xi=0.
\end{align*} As before, there are regular singular points at $0,
\infty$ and the $4n+2$ roots of $4n^2+y^{4n+2}=0$ which are
distinct and symmetrically distributed along a circle centred on
the origin of the complex plane. The $\beta$ coefficients are
\begin{align*}
  \beta_j=\left\{ \frac{(1+2n)(2n-5)}{4},\frac{5}{16},\ldots,\frac{5}{16}
  \right\},\quad\beta_\infty=0,
\end{align*} where $\tfrac{5}{16}$ appears $4n+2$ times. We note
that \begin{align*}   \sqrt{1+4\beta_0}=2\sqrt{n^2-2n-1}\ \notin
\mathbb{Q}\ \forall\ n \in \mathbb{N}.
\end{align*}  To show this we note
that $\sqrt{1+4\beta_0}\in\mathbb{Q} \Rightarrow n^2-2n-1=m^2$ for
some $m\in\mathbb{N}$. Since $n>m \Rightarrow m=n-\eta$ for
$\eta\in\mathbb{N}$ which leads to a quadratic in $\eta$ whose
roots are $(2n+\tfrac{1}{2},-\tfrac{1}{2})\notin\mathbb{N}$.
Therefore we can rule out cases $I$ and $III$ of the Kovacic
algorithm as in Section 3.

However, in analysing case $II$, the $E_j$ sets as in
\eqref{esets} are \begin{align*} E_0=\{2\},\quad
E_{1\ldots(4n+2)}=\{2,5,-1\},\quad E_\infty=\{0,2,4\}
\end{align*} and as such there will be a combination leading to
$d=\tfrac{1}{2}\big(4-(2-(4n+2)) \big)=2n+2$ and all values below.
Thus we can at best look at individual values of $n$, for example:

\begin{thm}
  The geodesic flow on the surface $x^2y^2z=1$ in not Louiville
  integrable with meromorphic first integrals.
\end{thm}

\begin{proof}
There are 615 combinations of the indicial exponents leading to
$d=0$; 55 leading to $d=1,2,3$ and 1 leading to $d=4,5,6$. Each of
these need to be check as described in Section 3. Again, the
procedure is tedious rather than difficult. As there are no
combinations for which the necessary $P$ exists, we can rule out
case $II$ of Kovacic's algorithm. Thus the identity component of
the differential Galois group of the normal variational equation
is not abelian, and therefore the geodesic flow on the surface
$z=1/(x^2-y^2)^2$ is not integrable.
\end{proof}

As the number of cases which need to be checked increases rapidly
for increasing $n$, the methods described in this paper are not
appropriate for testing the integrability of surfaces of the form
$x^ny^nz=1$. Having said that, since the $n=1$ and $n=2$ cases are
not integrable, we would conjecture that all surfaces of this form
with $n\in\mathbb{N}$ are also non-integrable.


\section{Conclusions}

Using Morales-Ramis theory and Kovacic's algorithm we are able to
rigorously prove the (meromorphic, Louiville) non-integrability of
the geodesic flow on certain algebraic surfaces. This approach is
very geometrical in flavour, as opposed to the topological
approach followed by Kozlov \cite{kozlov2}; nonetheless it is
robust enough to deal with free parameters and perturbations as
another paper by the author has shown \cite{TW}. The analysis was
facilitated by two features of the surfaces considered: Monge
patches allow a simple intrinsic coordinate
system/parameterization to be defined, and a plane of symmetry
leads to a planar geodesic along which the variational equations
decouple easily. It would be of interest to consider other
surfaces where these properties do not hold.


\section{Bibliography}

\bibliographystyle{plain}
\bibliography{geobib}

\end{document}